\documentclass[10pt,a4paper]{article}
\usepackage{indentfirst}
\usepackage[xdvi]{graphicx}
\usepackage[xdvi]{color}
\usepackage{latexsym,bm}
\usepackage{amssymb}
\usepackage{amsmath}
\usepackage{placeins}
\usepackage{amsthm}
\usepackage{epstopdf}

\newtheorem{thm}{Theorem}[section]

\newtheorem{lem}[thm]{Lemma}

\newtheorem{rem}[thm]{Remark}
\numberwithin{equation}{section}

\begin{document}
\title{Superconvergence  of both two and three dimensional rectangular Morley elements for biharmonic equations
\thanks{The first author was supported by  the NSFC Projects 11271035,  91430213 and 11421101.}}
\author{\normalsize Jun Hu$^\dagger$,~~Zhongci Shi$^\ast$,~~Xueqin Yang$^\dagger$\\ \normalsize
$^\dagger$ LMAM and School of Mathematical Sciences,
Peking University, \\ \normalsize Beijing 100871, P.R.China\\
$^\ast$ LESC, Institute of Computational Mathematics and Scientific/Engineering Computing,\\ Academy of Mathematics and System Science, Chinese Academy of Sciences,\\
\normalsize Beijing 100190, P.R.China
\\\vspace{2mm} \normalsize email: hujun@math.pku.edu.cn; ~~  shi@lsec.cc.ac.cn;\\yangxueqin1212@pku.edu.cn \normalsize
}
\date{}
\maketitle
\begin{abstract}
In the present paper, superconvergence of second order, after an appropriate postprocessing,
is achieved  for both the two and  three dimensional first order rectangular Morley elements of biharmonic equations.
The analysis is dependent on  superconvergence of second order for the consistency error and a corrected canonical interpolation operator, which
help to establish  supercloseness of second order  for the corrected canonical interpolation. Then the final superconvergence
 follows a  standard postprocessing. For first order nonconforming finite element methods
 of both two and three dimensional fourth order elliptic problems, it is the first time that  full superconvergence of second order is obtained without
 an extra boundary condition imposed on  exact solutions. It is also the first time that superconvergence is  established
 for  nonconforming finite element methods of three dimensional fourth order elliptic problems.  Numerical results are presented  to demonstrate the theoretical results.
\end{abstract}

\section{Introduction}
Because of significant  applications in scientific and engineering computing,   superconvergence analysis of finite element methods has become an active
  subject  since 70's last century.  However,  most of attentions have been paid  on  conforming and mixed finite element methods  of second order problems,
   we refer interested readers to \cite{Chen,ChenHuang,LinLinBook,LinYan1996} for more details. Since conforming finite element methods of
   fourth order problems  are very complicated,  most of popularly used elements in practice are  nonconforming, for instance,
   \cite{Lascaux85,Morley(1968),Shi1986,WangShiXu07,WangShiXub,WX06,WangXu07,WuMaoqing1983}.
   However, for nonconforming finite elements, due to nonconformity of
both trial and test functions, it becomes much more difficult to establish superconvergence
properties and related asymptotic error expansions. For second order elliptic problems, there are a few superconvergence results on rectangular elements. In \cite{ChenLi1994,ShiJiangB}, superconvergence  of the gradient  was obtained at the centers of elements for the Wilson element, which relies on the observation
 that the Wilson element space can be split into a conforming part and a nonconforming part.
  Due to superconvergence  of consistency errors, superconvergence of the nonconforming rotated $Q_1$ element \cite{RannacherTurek(1992)} and its variants was derived, see \cite{HuShi(2005),LinTobiska,MingShiXu2006}.  For the plate bending problem, there are only  few superconvergence results for nonconforming finite elements. In \cite{Chen}, Chen first established the supercloseness of the corrected interpolation of the incomplete biquadratic element \cite{WuMaoqing1983,Shi1986} on uniform rectangular meshes. By using similar corrected interpolations as in \cite{Chen}, Mao et al. \cite{MaoShi} first proved one and a half-order superconvergence for the Morley element \cite{Morley(1968)} and the incomplete biquadratic nonconforming element on uniform rectangular meshes.  In a recent paper \cite{HuMa}, Hu and Ma proposed a new method by using equivalence between the Morley element and the first order Hellan-Herrmann-Johnson element and obtained one and a half-order superconvergence for the Morley element on uniform mesh.  That  half order superconvergence
   can be improved to one order superconvergence  if the third order normal derivative of  exact solutions vanishes on the boundary of the domain under
    consideration. Based on the equivalence to the Stokes equations and a superconvergence result of Ye \cite{Ye2002} on the Crouzeix--Raivart element, Huang et al. \cite{HuangHuang2013} derived the superconvergence for the Morley element, which was postprocessed by projecting the finite element solution to another finite element space on a coarser mesh.  See Lin and Lin  \cite{LinLin}  for superconvergence of  the Ciarlet--Raviart scheme of  the biharmonic equation.
   Note that all of those results are only for fourth order problems in two dimensions. Superconvergence of nonconforming finite element methods cannot be found for fourth order problems in three dimensions.

    The purpose of the present paper is to analyze  superconvergence  of both  the two-dimensional and three-dimensional
     rectangular Morley elements from \cite{WangShiXu07}. Since  both of them are nonconforming, one difficulty is to
     bound the consistency error.  Another difficulty is from the canonical interpolation operator which does not admit
      supercloseness.  To overcome the first  difficulty, we use  some special orthogonal property of  the canonical interpolation
       operators of both the bilinear and trilinear elements when  applied to the functions in the  rectangular Morley element spaces.
       The other crucial observation is that the error between the (piecewise) gradient of functions in the discrete spaces and its mean  is equal
       on two opposite edges (faces) of an element.   In particular, this leads to superconvergence of  second order for the consistency error.
        To  deal with the second difficulty, we follow the idea from \cite{Chen}  to use  a correction of the canonical interpolation.
         Together with the asymptotic expansion results from \cite{HuYang}, this yields supercloseness of second order for such a corrected
          interpolation. Finally, based on the above  superconvergence results,  we  follow the postprocessing idea from \cite{LinYan1996}
            to  obtain  a  global superconvergent approximate solution, which converges at the second order convergence rate.
    It should be stressed that for first order nonconforming finite element methods
 of both two and three dimensional fourth order elliptic problems, it is the first time that  full superconvergence of  second order is obtained without
 an extra boundary condition imposed on  exact solutions. It is also the first time that superconvergence  is  established
 for  nonconforming finite element methods of three dimensional fourth order elliptic problems.

This paper is organized as follows. In the following section, we shall present the model problem and the rectangular Morley element. In section 3, we analyze the superconvergence property of the consistency error for the two-dimensional situation.  In section 4, we make a correction of the canonical interpolation and obtain the superconvergence result after the postprocessing. In section 5, we establish the superconvergence result for the three-dimensional cubic Morley element. In the last section 6, we present some numerical results to demonstrate our theoretical results.

\section{The model problem and the rectangular Morley element}

\subsection{The model problem}
We consider the model fourth order elliptic problem: Given $f\in L^2(\Omega)$, $\Omega\subset \mathbb{R}^2$ is a bounded Lipschitz domain,
\begin{eqnarray}
&&\left\{
\begin{array}{lll}
\Delta^2 u=f,\quad in \;\Omega,\\
u=\frac{\partial u}{\partial n}=0,\quad on\; \partial\Omega. \label{eq45}
\end{array}
\right.\\\nonumber
\end{eqnarray}
The variational formula of problem (\ref{eq45}) is to find $u\in V:=H^2_0(\Omega)$, such that
\begin{equation}
 a(u,v):=(\nabla^2 u,\nabla^2 v)_{L^2(\Omega)}=(f,v)_{L^2(\Omega)},\; \text{for  any}\;\; v \in V. \label{eq1}
\end{equation}
where $\nabla^2 u$ denotes the Hessian matrix of the function $u$.

\subsection{The two-dimensional rectangular Morley element }

To consider the discretization of $(\ref{eq1})$ by the rectangular Morley element method, let $\mathcal{T}_h$ be a regular uniform rectangular triangulation of the domain $\Omega$. Given $K\in \mathcal{T}_h$, 
let $(x_{1,c},x_{2,c})$ be the center of $K$, the meshsize $h$ and affine mapping:\\
\begin{equation} \label{eq2}  \xi_1=\frac{x_1-x_{1,c}}{h},\quad\xi_2=\frac{x_2-x_{2,c}}{h}\;,\; \text{for any} \;(x_1,x_2)\in K.\end{equation}
On element $K$, the shape function space of the rectangular Morley element from \cite{ShiWang} reads\\
\begin{equation}
P(K):=P_2(K)+\text{span}\{x_1^3,x_2^3\}, \label{eq3}
\end{equation}
here and throughout this paper, $P_l(K)$ denotes the space of polynomials of degree $\leq l$ over $K$. The nodal parameters are: for any $v\in C^1(K)$,
\begin{equation}
D(v):=\bigg(v(a_i),\;\;\frac{1}{|e_j|}\int_{e_j} \frac{\partial v}{\partial n_{e_j}}\;\mathrm{d}s\bigg),\;\;i,j=1,2,3,4, \label{eq4}
\end{equation}
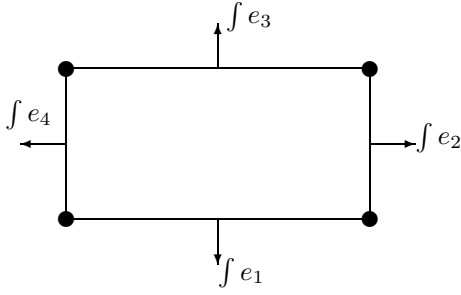
\begin{figure}
\begin{center}
\setlength{\unitlength}{2cm}
\begin{picture}(2,2)
\put(0,0.5){\line(1,0){2}} \put(0,0.5){\line(0,1){1}}
\put(2,0.5){\line(0,1){1}} \put(0,1.5){\line(1,0){2}}
\put(0,0.5){\circle*{0.1}}
\put(0,1.5){\circle*{0.1}}
\put(2,0.5){\circle*{0.1}}
\put(2,1.5){\circle*{0.1}}

\put(1,0.5){\vector(0,-1){0.3}}
\put(1,0.1){$\int e_1$}\put(1.05,1.8){$\int e_3$} \put(-0.4,1.15){$\int e_4$} \put(2.3,1){$\int e_2$}
\put(1,1.5){\vector(0,1){0.3}}
\put(0,1){\vector(-1,0){0.3}}
\put(2,1){\vector(1,0){0.3}}

\linethickness{0.6mm}
\end{picture}
\end{center}
\caption{degrees of freedom }
\end{figure}
where $a_i$ are vertices of $K$ and $e_j$ are edges with unit normal vectors $n_{e_j}$ of $K$, $|e_j|$ denote measure of edges $e_j$, see Figure 1.
 Let the reference element $\widehat{K}$ be a square on ($\xi_1$, $\xi_2$) plane,
 its vertices be $\widehat{a}_1(-1,-1)$, $\widehat{a}_2(1,-1)$, $\widehat{a}_3(1,1)$, $\widehat{a}_4(-1,1)$,
 and its sides be $\widehat{e}_1=\widehat{a}_1\widehat{a}_2$, $\widehat{e}_2=\widehat{a}_2\widehat{a}_3$,
 $\widehat{e}_3=\widehat{a}_3\widehat{a}_4$, $\widehat{e}_4=\widehat{a}_4\widehat{a}_1$.

The nonconforming rectangular Morley element space is then defined by
\begin{eqnarray*}
V_h&:=&\{v \in L^2(\Omega):\;v|_K \in P(K),\; \forall K \in \mathcal{T}_h, \;v\;is\;continuous\;at\;all\;internal\;vertices\;and\\&&\;vanishes\;at\;all\;boundary\;vertices,\;and\int_e \frac{\partial v}{\partial n_e}\,\mathrm{d}s\;is\;continuous\;on\;internal\;edges\;\\&& e\;and\;vanishes\;on\;boundary\;edges\; e\;of \;\mathcal{T}_h\}.
\end{eqnarray*}

The discrete problem of (\ref{eq1}) reads: Find $u_h \in V_h$, such that
\begin{equation}
a_h(u_h,v_h):=(\nabla_h^2 u_h,\nabla_h^2 v_h)_{L^2(\Omega)}=(f,v_h)_{L^2(\Omega)},\; \text{for any}\; v_h \in V_h. \label{eq5}
\end{equation}
where the operator $\nabla_h^2$ is the discrete counterpart of $\nabla^2$, which is defined element by element since the discrete space $V_h$ is nonconforming. Define a semi-norm over $V_h$ by
\begin{equation}
|u_h|_h^2:=a_h(u_h,u_h),\;\text{for any}\; u_h\in V_h.
\end{equation}
Let $u$ and $u_h$ be the solutions of $(\ref{eq1})$ and $(\ref{eq5})$, respectively, by the second Strang Lemma (\cite{Ciarlet},\cite{ShiWang}), we have
\begin{equation}
|u-u_h|_h\leq C\bigg(\inf\limits_{v_h\in V_h}|u-v_h|_h+\sup\limits_{0\neq w_h\in V_h}\frac{|a_h(u,w_h)-(f,w_h)|}{|w_h|_h}\bigg),
\end{equation}
where the first term is the approximation error and the second one is the consistency error. Herein and throughout this paper, $C$ denotes a generic positive constant which is independent of the meshsize and may be different at different places.

\section{Superconvergence  of the rectangular Morley element in 2D}
\subsection{Superconvergence  of the consistency error}
Let $I_h$ be piecewise bilinear interpolation operator on $\Omega$,  $I_h:V_h\rightarrow B_h$,
\begin{equation}
I_h v(P)=v(P),\;\text{for any vertex}\; P\;of\;\mathcal{T}_h,
\end{equation}
where
\begin{equation}
B_h=\big\{v\in H^1(\Omega),\; v|_K\in Q_1(K),\; \forall K\in \mathcal{T}_h\big\},
\end{equation}
and $Q_l(K)$ denotes the space of all polynomials which are of degree$\leq l$ with respect to each variable $x_i$, over $K$. Let
 the interpolation operator $\widehat{I}_{\widehat{K}}$ be the counterpart of $I_h$ on the reference element $\widehat{K}$.
The bilinear interpolation opertor $I_h$ has the following error estimate:
\begin{equation}
|v-I_hv|_{H^l(K)}\leq Ch^{2-l}|v|_{H^2(K)},\quad l=0,1, \label{eq57}
\end{equation}
for any  $v\in H^2(K)$. It is straightforward to see that $I_h$ is well defined for any $w_h\in V_h$. By Green's formula,
\begin{equation}
(f,I_h w_h)=(\Delta^2 u,I_h w_h)=-\int_\Omega \nabla\Delta u\cdot\nabla I_h w_h\;\mathrm{d}x_1\mathrm{d}x_2. \label{eq9}
\end{equation}
The integration by parts yields
\begin{eqnarray}
a_h(u,w_h)&=&-\sum_{K\in \mathcal{T}_h}\int_K\nabla\Delta u\cdot \nabla w_h\;\mathrm{d}x_1\mathrm{d}x_2+\sum_{K\in \mathcal{T}_h}\int_{\partial K}\frac{\partial^2 u}{\partial n^2}\frac{\partial w_h}{\partial n}\;\mathrm{d}s\nonumber\\&&+\sum_{K\in \mathcal{T}_h}\int_{\partial K}\frac{\partial ^2u}{\partial s\partial n}\frac{\partial w_h}{\partial s}\;\mathrm{d}s,\label{eq10}
\end{eqnarray}
where $\frac{\partial}{\partial s}$ and $\frac{\partial}{\partial n}$ are tangential and normal derivatives along element boundaries, respectively.
A combination of (\ref{eq9}) and (\ref{eq10}) yields
\begin{equation}
\begin{split}
a_h(u,w_h)-(f,w_h)&=a_h(u,w_h)-(f,I_h w_h)+(f,I_h w_h-w_h)\\& =-\sum_{K\in \mathcal{T}_h}\int_K\nabla\Delta u\cdot\nabla(w_h-I_h w_h)\;\mathrm{d}x_1\mathrm{d}x_2-\sum_{K\in \mathcal{T}_h}\int_Kf(w_h-I_h w_h)\;\mathrm{d}x_1\mathrm{d}x_2\\&+\sum_{K\in \mathcal{T}_h}\int_{\partial K}\frac{\partial^2 u}{\partial n^2}\frac{\partial w_h}{\partial n}\;\mathrm{d}s+\sum_{K\in \mathcal{T}_h}\int_{\partial K}\frac{\partial ^2u}{\partial s\partial n}\frac{\partial w_h}{\partial s}\;\mathrm{d}s.\label{eq11}
\end{split}
\end{equation}
The Cauchy-Schwarz inequality and the interpolation error estimate (\ref{eq57}) lead to
\begin{equation}
\bigg|\sum_{K\in\mathcal{T}_h}\int_Kf(w_h-I_h w_h)\;\mathrm{d}x_1\mathrm{d}x_2\bigg|\leq Ch^2||f||_{L^2(\Omega)}|w_h|_h, \label{eq14}
\end{equation}
which indicates a suprconvergence rate $O(h^2)$.

In the following three lemmas, we will analyze superconvergence for the three remaining terms of (\ref{eq11}).
\begin{lem}\label{lem1}
Suppose that $u\in H_0^2(\Omega)\bigcap H^4(\Omega)$ and $w_h\in V_h$. Then,
\begin{equation}
\sum_{K\in\mathcal{T}_h}\int_K\nabla\Delta u\cdot\nabla(w_h-I_h w_h)\;\mathrm{d}x_1\mathrm{d}x_2\leq Ch^2|u|_{H^4(\Omega)}|w_h|_h. \label{eq12}
\end{equation}
\end{lem}
\begin{proof}
On the reference element $\widehat{K}$, consider the following functional
\begin{equation}
B_1(\widehat{\phi},\widehat{w}_h)=\int_{\widehat{K}}\widehat{\phi}\frac{\partial (\widehat{w}_h-\widehat{I}_{\widehat{K}}\widehat{w}_h)}{\partial \xi_1}\;\mathrm{d}\xi_1\mathrm{d}\xi_2,
\end{equation}
A simple calculation leads to the interpolations, see Table 1.
\begin{table*}[h]
 \centering
\caption{calculation of interpolation}
 \begin{tabular}{cccccccccc}
 \hline
 $\widehat{w}_h$             &1          &$\xi_1$                  &$\xi_2$        &$\xi_1\xi_2$   &$\xi_1^2$&$\xi_2^2$&$\xi_1^3$&$\xi_2^3$&\\\hline
$ \widehat{I}_{\widehat{K}}\widehat{w}_h$   &1 &$\xi_1$ &$\xi_2$&$\xi_1\xi_2$&1&1&$\xi_1$&$\xi_2$&\\\hline
 \end{tabular}
\end{table*}
\\It can be checked that
\begin{eqnarray*}
&&\left\{
\begin{array}{lll}
B_1(\widehat{\phi},\widehat{w}_h)\leq 	C||\widehat{\phi}||_{L^2(\widehat{K})}|\widehat{w}_h|_{H^2(\widehat{K})},\\
B_1(\widehat{\phi},\widehat{w}_h)=0,\;\;\forall\widehat{\phi}\in P_0(\widehat{K}),\;\;\forall\widehat{w}_h\in V_h.
\end{array}
\right.\\
\end{eqnarray*}
The Bramble-Hilbert lemma gives
 \begin{equation}
 B_1(\widehat{\phi},\widehat{w}_h)\leq C\inf\limits_{\widehat{p}\in P_0(\widehat{K})}||\widehat{\phi}+\widehat{p}||_{L^2(\widehat{K})}|\widehat{w}_h|_{H^2(\widehat{K})}\leq C|\widehat{\phi}|_{H^1(\widehat{K})}|\widehat{w}_h|_{H^2(\widehat{K})}. \label{eq41}
 \end{equation}
 A substitution of $\phi=\frac{\partial \Delta u}{\partial x_1}$ into (\ref{eq41}), plus a scaling argument, yield
 \begin{equation}
 \int_K\frac{\partial \Delta u}{\partial x_1}\frac{\partial (w_h-I_h w_h)}{\partial x_1}\;\mathrm{d}x_1\mathrm{d}x_2\leq Ch^2|u|_{H^4(K)}|w_h|_h\;\;\text{for any}\;K\in\mathcal{T}_h.
 \end{equation}
 A similar argument proves
 \begin{equation}
 \int_K\frac{\partial \Delta u}{\partial x_2}\frac{\partial (w_h-I_h w_h)}{\partial x_2}\;\mathrm{d}x_1\mathrm{d}x_2\leq Ch^2|u|_{H^4(K)}|w_h|_h,
 \end{equation}
 which completes the proof.
\end{proof}
\begin{lem}\label{lem6}
Suppose that $u\in H_0^2(\Omega)\bigcap H^4(\Omega)$ and $w_h\in V_h$. Then,
\begin{equation}
\sum_{K\in \mathcal{T}_h}\int_{\partial K}\frac{\partial^2 u}{\partial n^2}\frac{\partial w_h}{\partial n}\;\mathrm{d}s\leq Ch^2|u|_{H^4(\Omega)}|w_h|_h.
\end{equation}
\end{lem}
\begin{proof}
Given $K\in \mathcal{T}_h$, let $e_i,i=1,\cdots,4$ be its four edges. Define $\Pi^0_{e_i}w=\frac{1}{|e_i|}\int_{e_i}w\mathrm{d}s$ and  $\mathcal{R}^0_{e_i}w=w-\Pi^0_{e_i}w$, for any $w\in L^2(K)$, then we have
\begin{equation}
\int_{e_i}\mathcal{R}^0_{e_i}w\;\mathrm{d}s=0.\label{eq53}
\end{equation}
Since $\int_{e_i}\frac{\partial w_h}{\partial n}\mathrm{d}s$ is continuous on internal edges $e_i$ and vanishes on boundary edges of $\mathcal{T}_h$, thus
\begin{eqnarray*}
\sum_{K\in \mathcal{T}_h}\int_{\partial K}\frac{\partial^2 u}{\partial n^2}\frac{\partial w_h}{\partial n}\;\mathrm{d}s&=&\sum_{K\in \mathcal{T}_h}\sum_{i=1}^4\int_{e_i}\frac{\partial^2 u}{\partial n^2}\frac{\partial w_h}{\partial n}\;\mathrm{d}s\\&=&\sum_{K\in \mathcal{T}_h}\sum_{i=1}^4\int_{e_i}\frac{\partial^2 u}{\partial n^2}\mathcal{R}^0_{e_i}\frac{\partial w_h}{\partial n}\;\mathrm{d}s\\&=&\sum_{K\in \mathcal{T}_h}\sum_{i=1}^4J_i.
\end{eqnarray*}
We first analyze the following terms
\begin{eqnarray*}
\sum_{K\in\mathcal{T}_h}J_2+J_4&=&\sum_{K\in\mathcal{T}_h}\bigg(\int_{e_2}\frac{\partial^2 u}{\partial n^2}\mathcal{R}^0_{e_2}\frac{\partial w_h}{\partial n}\;\mathrm{d}s+\int_{e_4}\frac{\partial^2 u}{\partial n^2}\mathcal{R}^0_{e_4}\frac{\partial w_h}{\partial n}\;\mathrm{d}s\bigg)\\&=&\sum_{K\in\mathcal{T}_h}\bigg(\int_{e_2}\frac{\partial^2 u}{\partial x_1^2}\mathcal{R}^0_{e_2}\frac{\partial w_h}{\partial x_1}\;\mathrm{d}x_2-\int_{e_4}\frac{\partial^2 u}{\partial x_1^2}\mathcal{R}^0_{e_4}\frac{\partial w_h}{\partial x_1}\;\mathrm{d}x_2\bigg)\\&=&\sum_{K\in\mathcal{T}_h}\int_{x_{2,c}-h}^{x_{2,c}+h}\bigg(\frac{\partial^2 u}{\partial x_1^2}\bigg|_{e_2}\mathcal{R}^0_{e_2}\frac{\partial w_h}{\partial x_1}\bigg|_{e_2}-\frac{\partial^2 u}{\partial x_1^2}\bigg|_{e_4}\mathcal{R}^0_{e_4}\frac{\partial w_h}{\partial x_1}\bigg|_{e_4}\bigg)\;\mathrm{d}x_2.
\end{eqnarray*}
For the rectangular Morley element, we have the following crucial property
\begin{equation*}
\mathcal{R}^0_{e_2}\frac{\partial w_h}{\partial x_1}\bigg|_{e_2}=\mathcal{R}^0_{e_4}\frac{\partial w_h}{\partial x_1}\bigg|_{e_4},\quad \mathcal{R}^0_{e_1}\frac{\partial w_h}{\partial x_2}\bigg|_{e_1}=\mathcal{R}^0_{e_3}\frac{\partial w_h}{\partial x_2}\bigg|_{e_3}.
\end{equation*}
This implies
\begin{eqnarray*}
\sum_{K\in\mathcal{T}_h}J_2+J_4&=&\sum_{K\in\mathcal{T}_h}\int_{x_{2,c}-h}^{x_{2,c}+h}\Big(\frac{\partial^2 u}{\partial x_1^2}\Big|_{e_2}-\frac{\partial^2 u}{\partial x_1^2}\Big|_{e_4}\Big)\mathcal{R}^0_{e_2}\frac{\partial w_h}{\partial x_1}\bigg|_{e_2}\;\mathrm{d}x_2\\&=&\sum_{K\in\mathcal{T}_h}\int_{x_{2,c}-h}^{x_{2,c}+h}\mathcal{R}^0_{e_2}\Big(\frac{\partial^2 u}{\partial x_1^2}\Big|_{e_2}-\frac{\partial^2 u}{\partial x_1^2}\Big|_{e_4}\Big)\mathcal{R}^0_{e_2}\frac{\partial w_h}{\partial x_1}\bigg|_{e_2}\;\mathrm{d}x_2.
\end{eqnarray*}
The error estimate of the interpolation operators $\Pi^0_{e_2}$ yields
\begin{equation}
\begin{split}
\sum_{K\in\mathcal{T}_h}J_2+J_4&\leq Ch\sum_{K\in\mathcal{T}_h}\Big|\Big|\nabla\Big(\frac{\partial^2 u}{\partial x_1^2}\Big|_{e_2}-\frac{\partial^2 u}{\partial x_1^2}\Big|_{e_4}\Big)\Big|\Big|_{L^2(K)}|w_h|_h\\&\leq Ch^2|u|_{H^4(\Omega)}|w_h|_h.\label{eq51}
\end{split}
\end{equation}
By the same argument, we can get
\begin{equation}
\sum_{K\in\mathcal{T}_h}J_1+J_3\leq Ch^2|u|_{H^4(\Omega)}|w_h|_h. \label{eq52}
\end{equation}
Then, a combination of (\ref{eq51}) and (\ref{eq52}) completes the proof.
\end{proof}
\begin{lem}\label{lem7}
Suppose that $u\in H_0^2(\Omega)\bigcap H^4(\Omega)$ and $w_h\in V_h$. Then,
\begin{equation}
\sum_{K\in \mathcal{T}_h}\int_{\partial K}\frac{\partial ^2u}{\partial s\partial n}\frac{\partial w_h}{\partial s}\;\mathrm{d}s\leq Ch^2|u|_{H^4(\Omega)}|w_h|_h.
\end{equation}
\end{lem}
The proof of this lemma can follow the similar procedure of Lemma \ref{lem6}, and herein we omit it.
According to above lemmas, we can obtain the following error estimate.
\begin{thm}\label{thm1}
Suppose that $u\in H^4(\Omega)$, $f\in L^2(\Omega)$ and for all $ w_h\in V_h$, then the consistency error can be estimated as
\begin{equation*}
a_h(u,w_h)-(f,w_h)\leq Ch^2(||f||_{L^2(\Omega)}+|u|_{H^4(\Omega)})|w_h|_h. \label{eq15}
\end{equation*}
\end{thm}
\subsection{Asymptotic expansion of  the canonical interpolation}
Given $K\in \mathcal{T}_h$, we define the canonical interpolation operator $\Pi_K : H^3(K)\rightarrow P(K)$ by, for any $v \in H^3(K) $,
\begin{equation}   \Pi_K v(P)=v(P)\; and\;\int_e \frac{\partial \Pi_K v}{\partial n_e}\;\mathrm{d}s=\int_e \frac{\partial v}{\partial n_e}\;\mathrm{d}s, \end{equation}\\
for any vertex $P$ of $K$ and any edge $e$ of $K$. The interpolation operator $\Pi_K$ has the following error estimates:
\begin{equation}  |v-\Pi_K v|_{H^l(K)}\leq Ch^{3-l}|v|_{H^3(K)},\;l=0,1,2,3, \end{equation}\\
provided that $v \in H^3(K)$. Then the global version $\Pi_h$ of the interpolation operator $\Pi_K$ is defined as
\begin{equation}  \Pi_h|_K=\Pi_K\;\; \text{for any}\; K \in \mathcal{T}_h. \end{equation}

We need the following asymptotic expansion result from  \cite{HuYang}.
\begin{lem}\label{lem10}
Suppose that $u\in H^4(\Omega)$, then for all $v_h\in V_h$, we have
\begin{eqnarray}
a_h(u-\Pi_h u,v_h)&\leq& \sum_{K\in\mathcal{T}_h} \frac{h^2}{3}\int_K\frac{\partial^3u}{\partial x_1\partial x_2^2}\frac{\partial^3v_h}{\partial x_1^3}\,\mathrm{d}x_1\mathrm{d}x_2\nonumber\\&&+\sum_{K\in\mathcal{T}_h} \frac{h^2}{3}\int_K\frac{\partial^3u}{\partial x_1^2\partial x_2}\frac{\partial^3v_h}{\partial x_2^3}\,\mathrm{d}x_1\mathrm{d}x_2\nonumber\\&&+Ch^2|u|_{H^4(\Omega)}|v_h|_h.\label{eq33}
\end{eqnarray}
\end{lem}
It is straightforward from Lemma \ref{lem10} to derive that by the inverse inequality
\begin{equation}
a_h(u-\Pi_h u,v_h)\leq Ch\Big(\Big|\Big|\frac{\partial^3u}{\partial x_1\partial x_2^2}\Big|\Big|_{L^2(\Omega)}+\Big|\Big|\frac{\partial^3u}{\partial x_1^2\partial x_2}\Big|\Big|_{L^2(\Omega)}\Big)|v_h|_h. \label{eq68}
\end{equation}

Based on the above analysis and Theorem \ref{thm1}, Lemma \ref{lem10}, we can get the following error estimate of $|\Pi_h u-u_h|_h$.
\begin{thm} \label{thm3}
Suppose that $u\in H^4(\Omega)$, then we have
\begin{equation}
|\Pi_h u-u_h|_h\leq Ch\Big(\Big|\Big|\frac{\partial^3u}{\partial x_1\partial x_2^2}\Big|\Big|_{L^2(\Omega)}+\Big|\Big|\frac{\partial^3u}{\partial x_1^2\partial x_2}\Big|\Big|_{L^2(\Omega)}\Big).
\end{equation}
\end{thm}
\begin{proof}
It follows from Theorem \ref{thm1} and (\ref{eq68}) that
\begin{eqnarray*}
|\Pi_h u-u_h|_h^2&=&a_h(\Pi_h u-u_h,\Pi_h u-u_h)\\&=&a_h(u-u_h,\Pi_h u-u_h)+a_h(\Pi_h u-u,\Pi_h u-u_h)\\&=&[a_h(u,\Pi_h u-u_h)-(f,\Pi_h u-u_h)]+a_h(\Pi_h u-u,\Pi_h u-u_h)\\&\leq&Ch^2(||f||_{L^2(\Omega)}+|u|_{H^4(\Omega)})|\Pi_h u-u_h|_h\\&&+Ch\Big(\Big|\Big|\frac{\partial^3u}{\partial x_1\partial x_2^2}\Big|\Big|_{L^2(\Omega)}+\Big|\Big|\frac{\partial^3u}{\partial x_1^2\partial x_2}\Big|\Big|_{L^2(\Omega)}\Big)|\Pi_h u-u_h|_h,
\end{eqnarray*}
which completes the proof.
\end{proof}
\section{Supercloseness  of the correction interpolation}
In the view of Theorem \ref{thm3}, we cannot expect a higher order error estimate of $|\Pi_h u-u_h|_h$. To overcome this difficulty, we follow the idea of \cite{Chen} to make a correction of the interpolation $\Pi_h u$.
First of all, we define the correction term as follows
\begin{equation}
R_Kv=\sum\limits_{j=5}^8a_j(v)\varphi_j,\quad v\in H^4(K),
\end{equation}
where $a_j(v)$ read
\begin{eqnarray}
&&\left\{
\begin{array}{lll}
a_j=-\frac{h}{6}\int_{e_2}\frac{\partial^3 v}{\partial x_1\partial x_2^2}\;\mathrm{d}x_2,\;\;j=5,7,\vspace{2mm}\\
a_j=-\frac{h}{6}\int_{e_1}\frac{\partial^3 v}{\partial x_1^2\partial x_2}\;\mathrm{d}x_1,\;\;j=6,8,\label{eq29}
\end{array}
\right.\\\nonumber
\end{eqnarray}
and the basis functions $\varphi_j$ read
\begin{eqnarray*}
&&\left\{
\begin{array}{lll}
\varphi_5=\frac{h}{4}(\xi_1+1)^2(\xi_1-1),\vspace{2mm}\\
\varphi_6=\frac{h}{4}(\xi_2+1)^2(\xi_2-1),\vspace{2mm}\\
\varphi_7=\frac{h}{4}(\xi_1+1)(\xi_1-1)^2,\vspace{2mm}\\
\varphi_8=\frac{h}{4}(\xi_2+1)(\xi_2-1)^2.
\end{array}
\right.\\
\end{eqnarray*}
$ \xi_1=\frac{x_1-x_{1,c}}{h},\;\;\xi_2=\frac{x_2-x_{2,c}}{h},$
where we defined in (\ref{eq2}).\\
Then the global version $R_h$ is defined as
\begin{equation}
 R_h|_K=R_K,\;\; \text{for any}\; K \in \mathcal{T}_h.
\end{equation}
Define the correction interpolation $\Pi_K^*v$  as follows, for all $v\in H^4(K)$,
\begin{equation}
\Pi_K^*v=\Pi_K v-R_Kv,\;\;K\in \mathcal{T}_h,
\end{equation}
Then the global version $\Pi_h^*$ is defined as
\begin{equation}
 \Pi_h^*|_K=\Pi^*_K,\;\; \text{for any}\; K \in \mathcal{T}_h.
\end{equation}
Regarding the correction term $R_h$, we have the following lemma.
\begin{lem}\label{lem5}
Suppose that $u\in H^4(\Omega)$, then for all $ v_h\in V_h$, we have
\begin{equation}
a_h(R_hu,v_h)=\frac{1}{2}\sum_{K\in \mathcal{T}_h}\int_K(a_5+a_7)\frac{\partial^3 v_h}{\partial x_1^3}\;\mathrm{d}x_1\mathrm{d}x_2+\frac{1}{2}\sum_{K\in \mathcal{T}_h}\int_K(a_6+a_8)\frac{\partial^3 v_h}{\partial x_2^3}\;\mathrm{d}x_1\mathrm{d}x_2. \label{eq34}
\end{equation}
\end{lem}
\begin{proof}
Let $\xi_1$ and $\xi_2 $  be defined as in $(\ref{eq2})$. It follows from the definition of $P(K)$ that
\begin{eqnarray}
\frac{\partial^2v_h}{\partial x_i^2}&=&\overline{\frac{\partial^2v_h}{\partial x_i^2}}+h\frac{\partial^3v_h}{\partial x_i^3}\xi_i,\quad i=1,2.
\end{eqnarray}
The definition of $a_h(\cdot,\cdot)$ yields
\begin{equation}
\begin{split}
a_h(R_hu,v_h)&=\sum_{K\in\mathcal{T}_h}\bigg(\int_K\frac{\partial^2 R_Ku}{\partial x_1^2}\frac{\partial^2 v_h}{\partial x_1^2}\;\mathrm{d}x_1\mathrm{d}x_2+2\int_K\frac{\partial^2 R_Ku}{\partial x_1\partial x_2}\frac{\partial^2 v_h}{\partial x_1\partial x_2}\;\mathrm{d}x_1\mathrm{d}x_2\\&\quad\;\;+\int_K\frac{\partial^2 R_Ku}{\partial x_2^2}\frac{\partial^2 v_h}{\partial x_2^2}\;\mathrm{d}x_1\mathrm{d}x_2\bigg).\label{eq26}
\end{split}
\end{equation}
We are in the position to calculate three terms on the right-hand side of (\ref{eq26}). It follows the definition of $R_K$ that
\begin{eqnarray*}
\int_K\frac{\partial^2 R_Ku}{\partial x_1^2}\frac{\partial^2 v_h}{\partial x_1^2}\;\mathrm{d}x_1\mathrm{d}x_2&=&\frac{h^{-1}}{4}\int_K\bigg[(6\xi_1+2)a_5+(6\xi_1-2)a_7\bigg]\frac{\partial^2 v_h}{\partial x_1^2}\;\mathrm{d}x_1\mathrm{d}x_2\nonumber\\&=&\frac{h^{-1}}{4}\int_K\bigg[6\xi_1(a_5+a_7)+2(a_5-a_7)\bigg]\bigg[\overline{\frac{\partial^2v_h}{\partial x_1^2}}+h\frac{\partial^3v_h}{\partial x_1^3}\xi_1\bigg]\;\mathrm{d}x_1\mathrm{d}x_2\nonumber\\&=&\frac{h^{-1}}{4}\int_K 6\xi_1(a_5+a_7)\overline{\frac{\partial^2v_h}{\partial x_1^2}}\;\mathrm{d}x_1\mathrm{d}x_2+\frac{h^{-1}}{4}\int_K6\xi_1^2(a_5+a_7)h\frac{\partial^3v_h}{\partial x_1^3}\;\mathrm{d}x_1\mathrm{d}x_2\nonumber\\&&+\frac{h^{-1}}{4}\int_K2(a_5-a_7)\overline{\frac{\partial^2v_h}{\partial x_1^2}}\;\mathrm{d}x_1\mathrm{d}x_2+\frac{h^{-1}}{4}\int_K2(a_5-a_7)h\frac{\partial^3v_h}{\partial x_1^3}\xi_1\;\mathrm{d}x_1\mathrm{d}x_2.
\end{eqnarray*}
Since coefficients like $\overline{\frac{\partial^2v_h}{\partial x_1^2}}$ and $\frac{\partial^3v_h}{\partial x_1^3}$ are constants, we can get that by parity of function and symmetry of domains:
\begin{equation*}
\int_K6\xi_1(a_5+a_7)\;\mathrm{d}x_1\mathrm{d}x_2=0,\quad \int_K2(a_5-a_7)h\xi_1\;\mathrm{d}x_1\mathrm{d}x_2=0.
\end{equation*}
Because of $a_5=a_7$, hence, only one nonzero term is left, which reads
\begin{equation}
\frac{h^{-1}}{4}\int_K6\xi_1^2(a_5+a_7)h\frac{\partial^3v_h}{\partial x_1^3}\;\mathrm{d}x_1\mathrm{d}x_2=\frac{1}{2}\int_K(a_5+a_7)\frac{\partial^3v_h}{\partial x_1^3}\;\mathrm{d}x_1\mathrm{d}x_2.\label{eq28}
\end{equation}
This yields
\begin{equation}
\int_K\frac{\partial^2 R_Ku}{\partial x_1^2}\frac{\partial^2 v_h}{\partial x_1^2}\;\mathrm{d}x_1\mathrm{d}x_2=\frac{1}{2}\int_K(a_5+a_7)\frac{\partial^3v_h}{\partial x_1^3}\;\mathrm{d}x_1\mathrm{d}x_2. \label{eq30}
\end{equation}
A similar argument proves
\begin{equation}
\int_K\frac{\partial^2 R_Ku}{\partial x_2^2}\frac{\partial^2 v_h}{\partial x_2^2}\;\mathrm{d}x_1\mathrm{d}x_2=\frac{1}{2}\int_K(a_6+a_8)\frac{\partial^3v_h}{\partial x_2^3}\;\mathrm{d}x_1\mathrm{d}x_2. \label{eq31}
\end{equation}
Note that the basis functions $\varphi_j(j=5,6,7,8)$ have no mixed terms, which leads to $\frac{\partial^2 R_K}{\partial x_1\partial x_2}=0$.
Thus
\begin{equation}
\int_K\frac{\partial^2 R_Ku}{\partial x_1\partial x_2}\frac{\partial^2 v_h}{\partial x_1\partial x_2}\;\mathrm{d}x_1\mathrm{d}x_2=0,
\end{equation}
which completes the proof.

\end{proof}
Based on the above analysis, we can establish superclose results of the rectangular Morley element by the correction interpolation $\Pi_h^*u$.
\begin{thm}\label{thm4}
Let $u\in H^4(\Omega)$, $u_h\in V_h$, be the solutions of (\ref{eq1}) and (\ref{eq5}), respectively, then we have
\begin{equation}
|\Pi_h^*u-u_h|_h\leq Ch^2\big(||f||_{L^2(\Omega)}+|u|_{H^4(\Omega)}\big).
\end{equation}
\end{thm}
\begin{rem}
Comparing with the incomplete biquadratic plate element \cite{MaoShi}, herein, the theorem does not require $\frac{\partial^3u}{\partial n^3}$ to be zero on the boundary. Besides, the correction interpolation $\Pi_h^*u$ still belongs to the space $V_h$. Because of the boundary condition $\frac{\partial u}{\partial n}\big|_{\partial \Omega}=0$,  it can be deduced that
$\frac{\partial u}{\partial x_1}\big|_{e_j}=0$ and $\frac{\partial u}{\partial x_2}\big|_{e_i}=0$, $e_i,e_j\in \partial \Omega,\;i=1,3,\;j=2,4$. Thus,
$\frac{\partial^3 u}{\partial x_1\partial x_2^2}\big|_{e_j}=0$ and $\frac{\partial^3 u}{\partial x_1^2\partial x_2}\big|_{e_i}=0$, $e_i,e_j\in \partial \Omega,\;i=1,3,\;j=2,4$.
\end{rem}
\begin{proof}
On the reference element $\widehat{K}$, consider the functional
\begin{eqnarray*}
B_2(\widehat{u},\widehat{v}_h)&=&\frac{1}{3}\int_{\widehat{K}}\frac{\partial^3\widehat{u}}{\partial\xi_1\partial\xi_2^2}\frac{\partial^3\widehat{v}_h}{\partial\xi_1^3}\;\mathrm{d}\xi_1\mathrm{d}\xi_2
+\frac{1}{3}\int_{\widehat{K}}\frac{\partial^3\widehat{u}}{\partial\xi_1^2\partial\xi_2}\frac{\partial^3\widehat{v}_h}{\partial\xi_2^3}\;\mathrm{d}\xi_1\mathrm{d}\xi_2\\&&-\frac{1}{6}\int_{\widehat{K}}
\bigg(\int_{\widehat{e}_2}\frac{\partial^3\widehat{u}}{\partial\xi_1\partial\xi_2^2}\;\mathrm{d}\xi_2\bigg)\frac{\partial^3\widehat{v}_h}{\partial\xi_1^3}\;\mathrm{d}\xi_1\mathrm{d}\xi_2-\frac{1}{6}\int_{\widehat{K}}
\bigg(\int_{\widehat{e}_1}\frac{\partial^3\widehat{u}}{\partial\xi_1^2\partial\xi_2}\;\mathrm{d}\xi_1\bigg)\frac{\partial^3\widehat{v}_h}{\partial\xi_2^3}\;\mathrm{d}\xi_1\mathrm{d}\xi_2
\end{eqnarray*}
It can be checked that
\begin{eqnarray*}
&&\left\{
\begin{array}{lll}
B_2(\widehat{u},\widehat{v}_h)\leq c||\widehat{u}||_{H^3(\widehat{K})}|\widehat{v}_h|_{H^2(\widehat{K})},\\
B_2(\widehat{u},\widehat{v}_h)=0,\;\;\forall\widehat{u}\in P_3(\widehat{K}),\;\;\forall\widehat{v}_h\in V_h.
\end{array}
\right.\\
\end{eqnarray*}
Hence, the Bramble-Hilbert lemma gives
 \begin{equation}
 B_2(\widehat{u},\widehat{v}_h)\leq C|\widehat{u}|_{H^4(\widehat{K})}|\widehat{v}_h|_{H^2(\widehat{K})}. \label{eq13}
 \end{equation}
 A scaling argument leads to
 \begin{equation*}
 \begin{split}
& \frac{h^2}{3}\int_K\frac{\partial^3u}{\partial x_1\partial x_2^2}\frac{\partial^3v_h}{\partial x_1^3}\;\mathrm{d} x_1\mathrm{d} x_2
+\frac{h^2}{3}\int_K\frac{\partial^3u}{\partial x_1^2\partial x_2}\frac{\partial^3v_h}{\partial x_2^3}\;\mathrm{d} x_1\mathrm{d} x_2\\&+\frac{1}{2}\int_K
(a_5+a_7)\frac{\partial^3v_h}{\partial x_1^3}\;\mathrm{d} x_1\mathrm{d}x_2+\frac{1}{2}\int_K
(a_6+a_8)\frac{\partial^3v_h}{\partial x_2^3}\;\mathrm{d}x_1\mathrm{d}x_2\\
 &\leq Ch^2|u|_{H^4(K)}|v_h|_h.
 \end{split}
 \end{equation*}
 An application of Lemma \ref{lem10} and Lemma \ref{lem5} yields
 \begin{equation}
 a_h(u-\Pi_h^*u,v_h)=a_h(u-\Pi_h u,v_h)+a_h(R_hu,v_h)\leq Ch^2|u|_{H^4(\Omega)}|v_h|_h. \label{eq32}
 \end{equation}
 Then, together with Theorem \ref{thm1} and (\ref{eq32}), this gives
 \begin{eqnarray*}
 |\Pi^*_h u-u_h|_h^2&=&a_h(\Pi^*_h u-u_h,\Pi^*_h u-u_h)\\&=&a_h(u-u_h,\Pi^*_h u-u_h)+a_h(\Pi^*_h u-u,\Pi^*_h u-u_h)\\&=&[a_h(u,\Pi^*_h u-u_h)-(f,\Pi^*_h u-u_h)]+a_h(\Pi^*_h u-u,\Pi^*_h u-u_h)\\&\leq&Ch^2(||f||_{L^2(\Omega)}+|u|_{H^4(\Omega)})|\Pi^*_h u-u_h|_h,
 \end{eqnarray*}
 which completes the proof.
\end{proof}
Based on the superclose property, we can obtain the superconvergence result of the two-dimensional rectangular Morley element by a proper postprocessing technique.  In order to attain the global superconvergence, we follow the idea of \cite{LinYan1996} to construct the postprocessing operator $\Pi_{3h}^3$ as follows.\\
We merge 9 adjacent elements into a macro element, $\widetilde{K}=\bigcup\limits_{i=1}^{9}K_i$, (see Figure 2 ), such that, in the macro element $\widetilde{K}$,
\begin{equation}
\Pi_{3h}^3w\in Q_3(\widetilde{K}),\quad \forall w \in C(\widetilde{K}).
\end{equation}
We denote $Z_{ij},i,j=1,2,3,4$ as the vertices of the 9 adjacent elements. Then, the operator $\Pi_{3h}^3$ satisfies
\begin{equation}
\Pi_{3h}^3w(Z_{ij})=w(Z_{ij}),\quad i,j=1,2,3,4.
\end{equation}
\begin{figure}
\begin{center}
\setlength{\unitlength}{1.2cm}
\begin{picture}(5,5)
\put(1,0){\line(1,0){3}}
\put(1,1){\line(1,0){3}}
\put(1,2){\line(1,0){3}}
\put(1,3){\line(1,0){3}}
\put(1,0){\line(0,1){3}}
\put(2,0){\line(0,1){3}}
\put(3,0){\line(0,1){3}}
\put(4,0){\line(0,1){3}}
\end{picture}
\end{center}
\caption{macro element $\widetilde{K}$}
\end{figure}
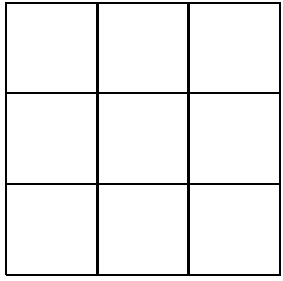
Besides, the postprocessing operator $\Pi_{3h}^3$ has the following properties
\begin{eqnarray}
&&\left\{
\begin{array}{lll}
\Pi_{3h}^3(\Pi^*_h u)=\Pi_{3h}^3u,\quad \forall u\in H^4(\Omega),\\
|\Pi_{3h}^3 v_h|_h\leq C|v_h|_h,\quad \forall v_h\in V_h,\\
|u-\Pi_{3h}^3 u|_h\leq Ch^2|u|_{H^4(\Omega)},\quad \forall u\in H^4(\Omega).\label{eq85}
\end{array}
\right.\\ \nonumber
\end{eqnarray}
Then, we can get the following global superconvergent result.
\begin{thm}
Let $u\in H^4(\Omega)$, $u_h\in V_h$, be the solutions of (\ref{eq1}) and (\ref{eq5}), respectively, then we have
\begin{equation}
|u-\Pi_{3h}^3u_h|_h\leq Ch^2\big(||f||_{L^2(\Omega)}+|u|_{H^4(\Omega)}\big).
\end{equation}
\end{thm}
\begin{proof}
It follows the properties (\ref{eq85}) and Theorem \ref{thm4} that
\begin{eqnarray}
|u-\Pi_{3h}^3u_h|_h&\leq&|u-\Pi_{3h}^3\Pi^*_h u_h|_h+|\Pi_{3h}^3(\Pi^*_h u-u_h)|_h\nonumber\\&\leq&|u-\Pi_{3h}^3u|_h+C|\Pi^*_h u-u_h|_h\nonumber\\&\leq&Ch^2\big(||f||_{L^2(\Omega)}+|u|_{H^4(\Omega)}\big),
\end{eqnarray}
which completes the proof.
\end{proof}

\section{Superconvergence  of the cubic Morley element}
In this section, we analyze the superconvergence property of the three-dimensional Morley element on cubic meshes with $\Omega\subset \mathbb{R}^3$.
Let $\mathcal{T}_h$ be a regular uniform cubic triangulation of the domain $\Omega\subset \mathbb{R}^3$. Given $K\in \mathcal{T}_h$, let $(x_{1,c},x_{2,c},x_{3,c})$ be the center of $K$, the meshsize $h$ and affine mapping:\\
\begin{equation}\label{eq72}  \xi_1=\frac{x_1-x_{1,c}}{h},\quad\xi_2=\frac{x_2-x_{2,c}}{h},\quad\xi_3=\frac{x_3-x_{3,c}}{h},\;\; \text{for any}\; (x_1,x_2,x_3)\in K.\end{equation}
On element $K$, the shape function space of the cubic Morley element reads\\
\begin{equation}   P(K):=P_2(K)+\text{span}\{x_1^3,x_2^3,x_3^3,x_1x_2x_3\}.\end{equation}
The nodal parameters are: for any $v\in C^1(K)$,
\begin{equation}
D(v)=\bigg(v(a_i),\;\;\frac{1}{|F_j|}\int_{F_j} \frac{\partial v}{\partial n_{F_j}}\;\mathrm{d}s\bigg),\;\;i=1,\ldots,8,\;j=1,\ldots,6,
\end{equation}
where $a_j$ are vertices of $K$ and $F_j$ are faces of $K$, see Figure 3.
\begin{figure}
\begin{center}
\setlength{\unitlength}{2.5cm}
\begin{picture}(4,2)

\put(1,0.5){\line(1,0){2}} \put(1,0.5){\line(0,1){1}}
\put(3,0.5){\line(0,1){1}} \put(1,1.5){\line(1,0){2}}
\put(1,0.5){\circle*{0.1}}

\put(1,1.5){\circle*{0.1}}
\put(3,0.5){\circle*{0.1}}
\put(3,1.5){\circle*{0.1}}
\put(1,1.5){\line(1,1){0.3}}
\put(3,1.5){\line(1,1){0.3}}
\put(1.3,1.8){\line(1,0){2}}
\put(1.3,1.8){\circle*{0.1}}
\put(3.3,1.8){\circle*{0.1}}
\put(3,0.5){\line(1,1){0.3}}
\put(3.3,0.8){\circle*{0.1}}
\put(3.3,0.8){\line(0,1){1}}
\multiput(1,0.5)(0.18,0.18){2}{\line(1,1){0.15}}
\multiput(1.3,0.8)(0.14,0){15}{\line(1,0){0.1}}
\multiput(1.3,0.8)(0,0.14){7}{\line(0,1){0.1}}
\put(1.3,0.8){\circle*{0.1}}
\put(2.2,0.65){\vector(0,-1){0.3}}
\put(2.2,1.65){\vector(0,1){0.3}}
\put(1.15,1.15){\vector(-1,0){0.3}}
\put(3.15,1.15){\vector(1,0){0.3}}
\put(2,1){\vector(-1,-1){0.25}}
\put(2.3,1.3){\vector(1,1){0.25}}
\put(2.2,0.25){$\iint F_6$}
\put(2.2,1.95){$\iint F_3$}
\put(0.6,1.2){$\iint F_5$}
\put(3.45,1.15){$\iint F_2$}
\put(1.4,0.8){$\iint F_1$}
\put(2.6,1.6){$\iint F_4$}
\linethickness{0.6mm}
\end{picture}
\end{center}
\caption{degrees of freedom }
\end{figure}
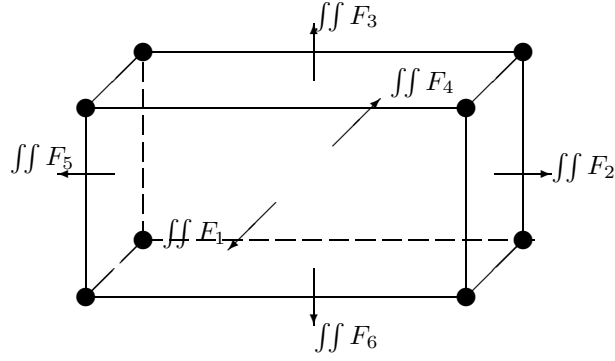
\subsection{Superconvergence  of the consistency error}
We also need the decomposition of the consistency error (\ref{eq11}). For ease of reading, we recall the expression as follows
\begin{eqnarray}
a_h(u,w_h)-(f,w_h)&=&a_h(u,w_h)-(f,I_h w_h)+(f,I_h w_h-w_h) \nonumber\\&=&-\sum_{K\in \mathcal{T}_h}\int_K\nabla\Delta u\cdot\nabla(w_h-I_h w_h)\;\mathrm{d}x_1\mathrm{d}x_2\mathrm{d}x_3\nonumber\\&&-\sum_{K\in \mathcal{T}_h}\int_K f(w_h-I_h w_h)\;\mathrm{d}x_1\mathrm{d}x_2\mathrm{d}x_3+\sum_{K\in \mathcal{T}_h}\sum_{i=1}^3\int_{\partial K}\frac{\partial^2 u}{\partial x_i^2}\frac{\partial w_h}{\partial x_i}n_i\;\mathrm{d}s\nonumber\\&&+\sum_{K\in \mathcal{T}_h}\sum_{1\leq i\neq j\leq3}\int_{\partial K}\frac{\partial^2 u}{\partial x_i\partial x_j}\frac{\partial w_h}{\partial x_j}n_i\;\mathrm{d}s.\label{eq54}
\end{eqnarray}
A direct application of the interpolation error estimate (\ref{eq57}) leads to
\begin{equation}
\bigg|\sum_{K\in\mathcal{T}_h}\int_K f(w_h-I_h w_h)\;\mathrm{d}x_1\mathrm{d}x_2\bigg|\leq Ch^2||f||_{L^2(\Omega)}|w_h|_h.\label{eq36}
\end{equation}
\begin{lem}\label{lem2}
Suppose that $u\in H_0^2(\Omega)\bigcap H^4(\Omega)$ and $w_h\in V_h$. Then,
\begin{equation*}
\sum_{K\in\mathcal{T}_h}\int_K\nabla\Delta u\cdot\nabla(w_h-I_h w_h)\;\mathrm{d}x_1\mathrm{d}x_2\mathrm{d}x_3\leq Ch^2|u|_{H^4(\Omega)}|w_h|_h.
\end{equation*}
\end{lem}
\begin{proof}
On the reference element $\widehat{K}$, consider the following functional
\begin{equation}
B_3(\widehat{\phi},\widehat{w}_h)=\int_{\widehat{K}}\widehat{\phi}\frac{\partial (\widehat{w}_h-\widehat{I}_{\widehat{K}}\widehat{w}_h)}{\partial \xi_1}\;\mathrm{d}\xi_1\mathrm{d}\xi_2\mathrm{d}\xi_3,
\end{equation}
A simple calculation leads to the interpolations, see Table 2.
\begin{table*}[h]
 \centering
\caption{calculation of interpolation}
 \begin{tabular}{ccccccccccccccc}
 \hline
 $\widehat{w}_h$             &1          &$\xi_1$                  &$\xi_2$   &$\xi_3$     &$\xi_1\xi_2$  &$\xi_1\xi_3$&$\xi_2\xi_3$ &$\xi_1^2$&$\xi_2^2$&$\xi_3^2$&$\xi_1^3$&$\xi_2^3$&$\xi_3^3$&\\\hline
$ \widehat{I}_{\widehat{K}}\widehat{w}_h$   &1 &$\xi_1$ &$\xi_2$&$\xi_3$&$\xi_1\xi_2$&$\xi_1\xi_3$&$\xi_2\xi_3$&1&1&1&$\xi_1$&$\xi_2$&$\xi_3$&\\\hline
 \end{tabular}
\end{table*}
\\It follows that
\begin{eqnarray*}
&&\left\{
\begin{array}{lll}
B_3(\widehat{\phi},\widehat{w}_h)\leq c||\widehat{\phi}||_{L^2(\widehat{K})}|\widehat{w}_h|_{H^2(\widehat{K})},\\
B_3(\widehat{\phi},\widehat{w}_h)=0,\;\;\forall\widehat{\phi}\in P_0(\widehat{K}),\;\;\forall\widehat{w}_h\in V_h.
\end{array}
\right.\\
\end{eqnarray*}
The Bramble-Hilbert lemma gives
 \begin{equation}
 B_3(\widehat{\phi},\widehat{w}_h)\leq C\inf\limits_{\widehat{p}\in P_0(\widehat{K})}||\widehat{\phi}+\widehat{p}||_{L^2(\widehat{K})}|\widehat{w}_h|_{H^2(\widehat{K})}\leq C|\widehat{\phi}|_{H^1(\widehat{K})}|\widehat{w}_h|_{H^2(\widehat{K})}. \label{eq58}
 \end{equation}
 A substitution of $\phi=\frac{\partial \Delta u}{\partial x_1}$ into (\ref{eq58}), plus a scaling argument yield
 \begin{equation}
 \int_K\frac{\partial \Delta u}{\partial x_1}\frac{\partial (w_h-I_h w_h)}{\partial x_1}\;\mathrm{d}x_1\mathrm{d}x_2\mathrm{d}x_3\leq Ch^2|u|_{H^4(K)}|w_h|_h,\;\text{for any}\;K\in\mathcal{T}_h.
 \end{equation}
 A similar argument proves
 \begin{equation}
 \int_K\frac{\partial \Delta u}{\partial x_2}\frac{\partial (w_h-I_h w_h)}{\partial x_2}\;\mathrm{d}x_1\mathrm{d}x_2\mathrm{d}x_3\leq Ch^2|u|_{H^4(K)}|w_h|_h,
 \end{equation}
 and
 \begin{equation}
 \int_K\frac{\partial \Delta u}{\partial x_3}\frac{\partial (w_h-I_h w_h)}{\partial x_3}\;\mathrm{d}x_1\mathrm{d}x_2\mathrm{d}x_3\leq Ch^2|u|_{H^4(K)}|w_h|_h,
 \end{equation}
 which complete the proof.
\end{proof}

Next, we will analyze the last two terms of (\ref{eq54}).
\begin{lem}\label{lem3}
Suppose that $u\in H_0^2(\Omega)\bigcap H^4(\Omega)$ and $w_h\in V_h$. Then,
\begin{equation*}
\sum_{K\in \mathcal{T}_h}\sum_{i=1}^3\int_{\partial K}\frac{\partial^2 u}{\partial x_i^2}\frac{\partial w_h}{\partial x_i}n_i\;\mathrm{d}s\nonumber\leq Ch^2|u|_{H^4(\Omega)}|w_h|_h.
\end{equation*}
\end{lem}
\begin{proof}
Given $K\in \mathcal{T}_h$, let $F_i,i=1,\cdots,6$ be its faces. Define $\Pi^0_{F_i}w=\frac{1}{|F_i|}\int_{F_i}w\mathrm{d}s$ and $\mathcal{R}^0_{F_i}w=w-\Pi^0_{F_i}w$, for any $w\in L^2(K)$, then we have
\begin{equation}
\int_{F_i}\mathcal{R}^0_{F_i}w\;\mathrm{d}s=0.\label{37}
\end{equation}
Since $\int_{F_i}\frac{\partial w_h}{\partial n}\mathrm{d}s$ is continuous on internal faces $F_i$ and vanishes on boundary faces of $\mathcal{T}_h$, thus
\begin{eqnarray*}
\sum_{K\in \mathcal{T}_h}\sum_{i=1}^3\int_{\partial K}\frac{\partial^2 u}{\partial x_i^2}\frac{\partial w_h}{\partial x_i}n_i\;\mathrm{d}s\nonumber&=&\sum_{K\in \mathcal{T}_h}\sum_{i=1}^3\sum_{j=1}^6\int_{F_j}\frac{\partial^2 u}{\partial x_i^2}\frac{\partial w_h}{\partial x_i}n_i\;\mathrm{d}s\\&=&\sum_{K\in \mathcal{T}_h}\sum_{i=1}^3\sum_{j=1}^6\int_{F_j}\frac{\partial^2 u}{\partial x_i^2}\mathcal{R}^0_{F_j}\frac{\partial w_h}{\partial x_i}n_i\;\mathrm{d}s
\end{eqnarray*}
For ease of expression, denote $L_j n_i=\int_{F_j}\frac{\partial^2 u}{\partial x_i^2}\mathcal{R}^0_{F_j}\frac{\partial w_h}{\partial x_i}n_i\;\mathrm{d}s,\;i=1,2,3,\;j=1,\cdots,6$. \\ Then we firstly analyze the following terms
\begin{eqnarray*}
&&\sum_{K\in\mathcal{T}_h}\sum_{i=1}^3(L_2+L_5)n_i\\&&=\sum_{K\in\mathcal{T}_h}\int_{F_2}\frac{\partial^2 u}{\partial x_2^2}\mathcal{R}^0_{F_2}\frac{\partial w_h}{\partial x_2}\;\mathrm{d}x_1\mathrm{d}x_3-\sum_{K\in\mathcal{T}_h}\int_{F_5}\frac{\partial^2 u}{\partial x_2^2}\mathcal{R}^0_{F_5}\frac{\partial w_h}{\partial x_2}\;\mathrm{d}x_1\mathrm{d}x_3\\&&=\sum_{K\in\mathcal{T}_h}\int_{x_{1,c}-h}^{x_{1,c}+h}\int_{x_{3,c}-h}^{x_{3,c}+h}\bigg(\frac{\partial^2 u}{\partial x_2^2}\bigg|_{F_2}\mathcal{R}^0_{F_2}\frac{\partial w_h}{\partial x_2}\bigg|_{F_2}-\frac{\partial^2 u}{\partial x_2^2}\bigg|_{F_5}\mathcal{R}^0_{F_5}\frac{\partial w_h}{\partial x_2}\bigg|_{F_5}\bigg)\;\mathrm{d}x_1\mathrm{d}x_3.
\end{eqnarray*}
For the cubic Morley element, we have the following crucial property
\begin{equation*}
\mathcal{R}^0_{F_1}\frac{\partial w_h}{\partial x_1}\Big|_{F_1}=\mathcal{R}^0_{F_4}\frac{\partial w_h}{\partial x_1}\Big|_{F_4},\;\mathcal{R}^0_{F_2}\frac{\partial w_h}{\partial x_2}\Big|_{F_2}=\mathcal{R}^0_{F_5}\frac{\partial w_h}{\partial x_2}\Big|_{F_5},\;\mathcal{R}^0_{F_3}\frac{\partial w_h}{\partial x_3}\Big|_{F_3}=\mathcal{R}^0_{F_6}\frac{\partial w_h}{\partial x_3}\Big|_{F_6}.
\end{equation*}
This implies
\begin{eqnarray*}
\sum_{K\in\mathcal{T}_h}\sum_{i=1}^3(L_2+L_5)n_i&=&\sum_{K\in\mathcal{T}_h}\int_{x_{1,c}-h}^{x_{1,c}+h}\int_{x_{3,c}-h}^{x_{3,c}+h}\bigg(\frac{\partial^2 u}{\partial x_2^2}\bigg|_{F_2}-\frac{\partial^2 u}{\partial x_2^2}\bigg|_{F_5}\bigg)\mathcal{R}^0_{F_2}\frac{\partial w_h}{\partial x_2}\Big|_{F_2}\mathrm{d}x_1\mathrm{d}x_3\\&=&\sum_{K\in\mathcal{T}_h}\int_{x_{1,c}-h}^{x_{1,c}+h}\int_{x_{3,c}-h}^{x_{3,c}+h}\mathcal{R}^0_{F_2}\bigg(\frac{\partial^2 u}{\partial x_2^2}\bigg|_{F_2}-\frac{\partial^2 u}{\partial x_2^2}\bigg|_{F_5}\bigg)\mathcal{R}^0_{F_2}\frac{\partial w_h}{\partial x_2}\Big|_{F_2}\mathrm{d}x_1\mathrm{d}x_3.
\end{eqnarray*}
The error estimate of the interpolation operators $\Pi^0_{F_2}$ yields
\begin{equation}
\begin{split} 
\sum_{K\in\mathcal{T}_h}\sum_{i=1}^3(L_2+L_5)n_i&\leq Ch\sum_{K\in\mathcal{T}_h}\Big|\Big|\nabla\Big(\frac{\partial^2 u}{\partial x_2^2}\Big|_{F_2}-\frac{\partial^2 u}{\partial x_2^2}\Big|_{F_5}\Big)\Big|\Big|_{L^2(K)}|w_h|_h\\&\leq Ch^2|u|_{H^4(\Omega)}|w_h|_h.\label{eq44}
\end{split}
\end{equation}
A similar argument proves
\begin{equation}
\sum_{K\in\mathcal{T}_h}\sum_{i=1}^3(L_1+L_4)n_i\leq Ch^2|u|_{H^4(\Omega)}|w_h|_h, \label{eq46}
\end{equation}
and
\begin{equation}
\sum_{K\in\mathcal{T}_h}\sum_{i=1}^3(L_3+L_6)n_i\leq Ch^2|u|_{H^4(\Omega)}|w_h|_h. \label{eq47}
\end{equation}
Then, a combination of (\ref{eq44}), (\ref{eq46}) and (\ref{eq47}) completes the proof.
\end{proof}
\begin{lem}\label{lem4}
Suppose that $u\in H_0^2(\Omega)\bigcap H^4(\Omega)$ and $w_h\in V_h$. Then,
\begin{equation}
\sum_{K\in \mathcal{T}_h}\sum_{1\leq i\neq j\leq3}\int_{\partial K}\frac{\partial^2 u}{\partial x_i\partial x_j}\frac{\partial w_h}{\partial x_j}n_i\;\mathrm{d}s\leq Ch^2|u|_{H^4(\Omega)}|w_h|_h.
\end{equation}
\end{lem}
The proof of this lemma can follow the similar procedure of Lemma \ref{lem3}, and herein we omit it. According to above lemmas, we can obtain the following error estimate.
\begin{thm}\label{thm8}
Suppose that $u\in H^4(\Omega)$, $f\in L^2(\Omega)$. Then it holds that
\begin{equation*}
a_h(u,w_h)-(f,w_h)\leq Ch^2(||f||_{L^2(\Omega)}+|u|_{H^4(\Omega)})|w_h|_h,\;\;\text{for any}\;w_h\in V_h.
\end{equation*}
\end{thm}

 We need the following asymptotic expansion result of the  canonical interpolation  from \cite{HuYang}.
\begin{lem}\label{lem11}
Suppose that $u\in H^4(\Omega)$, $\Omega\subset \mathbb{R}^3$, then for all $v_h\in V_h$, we have
\begin{eqnarray*}
a_h(u-\Pi_h u,v_h)&\leq& \sum_{K\in\mathcal{T}_h}\sum_{i\neq j=1}^3 \frac{h^2}{3}\int_K\frac{\partial^3u}{\partial x_i\partial x_j^2}\frac{\partial^3v_h}{\partial x_i^3}\,\mathrm{d}x_1\mathrm{d}x_2\mathrm{d}x_3+Ch^2|u|_{H^4(\Omega)}|v_h|_h.
\end{eqnarray*}
\end{lem}
It is straightforward from Lemma \ref{lem11} to derive that by the inverse inequality
\begin{equation}
a_h(u-\Pi_h u,v_h)\leq Ch\sum_{i\neq j=1}^3\Big|\Big|\frac{\partial^3u}{\partial x_i\partial x_j^2}\Big|\Big|_{L^2(\Omega)}|v_h|_h.\label{eq69}
\end{equation}
Based on the analysis of the interpolation error and the consistency error, we can get the following error estimate of $|\Pi_h u-u_h|_h$.
\begin{thm} \label{thm6}
Suppose that $u\in H^4(\Omega)$, then we have
\begin{equation*}
|\Pi_h u-u_h|_h\leq Ch\sum_{i\neq j=1}^3\Big|\Big|\frac{\partial^3u}{\partial x_i\partial x_j^2}\Big|\Big|_{L^2(\Omega)}.
\end{equation*}
\end{thm}
\begin{proof}
It follows from Theorem \ref{thm8} and (\ref{eq69}) that
\begin{eqnarray*}
|\Pi_h u-u_h|_h^2&=&a_h(\Pi_h u-u_h,\Pi_h u-u_h)\\&=&a_h(u-u_h,\Pi_h u-u_h)+a_h(\Pi_h u-u,\Pi_h u-u_h)\\&=&[a_h(u,\Pi_h u-u_h)-(f,\Pi_h u-u_h)]+a_h(\Pi_h u-u,\Pi_h u-u_h)\\&\leq&Ch^2(||f||_{L^2(\Omega)}+|u|_{H^4(\Omega)})|\Pi_h u-u_h|_h+Ch\sum_{i\neq j=1}^3\Big|\Big|\frac{\partial^3u}{\partial x_i\partial x_j^2}\Big|\Big|_{L^2(\Omega)}|\Pi_h u-u_h|_h.
\end{eqnarray*}
which completes the proof.
\end{proof}
\subsection{Supercloseness  of the correction interpolation}
We can learn from Theorem \ref{thm6} that the convergence of the error $|\Pi_h u-u_h|_h$ is only of order $O(h)$. Therefore, we follow the idea of \cite{Chen} to make a correction of the interpolation to improve its convergence. The operator $\widetilde{\Pi}_K^{*}$ is modified as
\begin{equation}
\widetilde{\Pi}_K^{*}u=\Pi_K u-\widetilde{R}_Ku,\quad u\in H^4(K),
\end{equation}
where $\widetilde{R}_Ku=\sum\limits_{j=9}^{14}b_j(u)\widetilde{\varphi}_j$ with
\begin{eqnarray}
&&\left\{
\begin{array}{lll}
b_j=-\frac{1}{12}\int_{F_1}\Big(\frac{\partial^3 v}{\partial x_1\partial x_2^2}+\frac{\partial^3 v}{\partial x_1\partial x_3^2}\Big)\;\mathrm{d}x_2\mathrm{d}x_3,\;\;j=9,12,\vspace{2mm}\\
b_j=-\frac{1}{12}\int_{F_2}\Big(\frac{\partial^3 v}{\partial x_2\partial x_1^2}+\frac{\partial^3 v}{\partial x_2\partial x_3^2}\Big)\;\mathrm{d}x_1\mathrm{d}x_3,\;\;j=10,13,\vspace{2mm}\\
b_j=-\frac{1}{12}\int_{F_3}\Big(\frac{\partial^3 v}{\partial x_3\partial x_1^2}+\frac{\partial^3 v}{\partial x_3\partial x_2^2}\Big)\;\mathrm{d}x_1\mathrm{d}x_2,\;\;j=11,14,\label{eq55}
\end{array}
\right.\\\nonumber
\end{eqnarray}
and the basis functions
\begin{eqnarray}
&&\left\{
\begin{array}{lll}
\widetilde{\varphi}_9=\frac{h}{4}(\xi_1+1)^2(\xi_1-1),\vspace{2mm}\\
\widetilde{\varphi}_{10}=\frac{h}{4}(\xi_2+1)^2(\xi_2-1),\vspace{2mm}\\
\widetilde{\varphi}_{11}=\frac{h}{4}(\xi_3+1)^2(\xi_3-1),\vspace{2mm}\\
\widetilde{\varphi}_{12}=\frac{h}{4}(\xi_1+1)(\xi_1-1)^2,\vspace{2mm}\\
\widetilde{\varphi}_{13}=\frac{h}{4}(\xi_2+1)(\xi_2-1)^2,\vspace{2mm}\\
\widetilde{\varphi}_{14}=\frac{h}{4}(\xi_3+1)(\xi_3-1)^2.
\end{array}
\right.\\\nonumber
\end{eqnarray}
$\xi_1=\frac{x_1-x_{1,c}}{h},\;\xi_2=\frac{x_2-x_{2,c}}{h},\;\xi_3=\frac{x_3-x_{3,c}}{h},$ where we defined in (\ref{eq72}).\\
Then the global version $\widetilde{\Pi}_h^{*}$ and $\widetilde{R}_h$ are defined as
\begin{equation*}
\widetilde{R}_h|_K=\widetilde{R}_K,\quad \text{for any}\;K\in \mathcal{T}_h,
\end{equation*}
\begin{equation*}
\widetilde{\Pi}_h^{*}|_K=\widetilde{\Pi}_K^{*},\quad \text{for any}\;K\in \mathcal{T}_h.
\end{equation*}
Thus, we can estabilsh superclose results of the three-dimensional cubic Morley element by the correction interpolation $\widetilde{\Pi}_h^{*}u$.
\begin{thm}\label{thm7}
Let $u\in H^4(\Omega)$, $u_h\in V_h$, be the solutions of (\ref{eq1}) and (\ref{eq5}), respectively, then it holds
\begin{equation}
|\widetilde{\Pi}_h^*u-u_h|_h\leq Ch^2(||f||_{L^2(\Omega)}+|u|_{H^4(\Omega)}).
\end{equation}
\end{thm}
\begin{proof}
On the reference element $\widehat{K}$, consider the functional
\begin{eqnarray*}
B_4(\widehat{u},\widehat{v}_h)&=&\sum_{K\in\mathcal{T}_h}\sum_{i\neq j=1}^3 \frac{1}{3}\int_{\widehat{K}}\frac{\partial^3\widehat{u}}{\partial \xi_i\partial \xi_j^2}\frac{\partial^3\widehat{v}_h}{\partial \xi_i^3}\,\mathrm{d}\xi_1\mathrm{d}\xi_2\mathrm{d}\xi_3\\&&-\frac{1}{12}\int_{\widehat{K}}
\bigg(\int_{\widehat{F}_1}\bigg(\frac{\partial^3\widehat{u}}{\partial\xi_1\partial\xi_2^2}+\frac{\partial^3\widehat{u}}{\partial\xi_1\partial\xi_3^2}\bigg)\;\mathrm{d}\xi_2\mathrm{d}\xi_3\bigg)
\frac{\partial^3\widehat{v}_h}{\partial\xi_1^3}\;\mathrm{d}\xi_1\mathrm{d}\xi_2\mathrm{d}\xi_3\\&&-\frac{1}{12}\int_{\widehat{K}}
\bigg(\int_{\widehat{F}_2}\bigg(\frac{\partial^3\widehat{u}}{\partial\xi_1^2\partial\xi_2}+\frac{\partial^3\widehat{u}}{\partial\xi_2\partial\xi_3^2}\bigg)\;\mathrm{d}\xi_1\mathrm{d}\xi_3\bigg)
\frac{\partial^3\widehat{v}_h}{\partial\xi_2^3}\;\mathrm{d}\xi_1\mathrm{d}\xi_2\mathrm{d}\xi_3\\&&-\frac{1}{12}\int_{\widehat{K}}
\bigg(\int_{\widehat{F}_3}\bigg(\frac{\partial^3\widehat{u}}{\partial\xi_1^2\partial\xi_3}+\frac{\partial^3\widehat{u}}{\partial\xi_2^2\partial\xi_3}\bigg)\;\mathrm{d}\xi_1\mathrm{d}\xi_2\bigg)
\frac{\partial^3\widehat{v}_h}{\partial\xi_3^3}\;\mathrm{d}\xi_1\mathrm{d}\xi_2\mathrm{d}\xi_3.
\end{eqnarray*}
It can be checked that
\begin{eqnarray*}
&&\left\{
\begin{array}{lll}
B_4(\widehat{u},\widehat{v}_h)\leq c||\widehat{u}||_{H^3(\widehat{K})}|\widehat{v}_h|_{H^2(\widehat{K})},\\
B_4(\widehat{u},\widehat{v}_h)=0,\;\;\forall\widehat{u}\in P_3(\widehat{K}),\;\;\forall\widehat{v}_h\in V_h.
\end{array}
\right.\\
\end{eqnarray*}
Hence, the Bramble-Hilbert lemma gives
 \begin{equation}
 B_4(\widehat{u},\widehat{v}_h)\leq C|\widehat{u}|_{H^4(\widehat{K})}|\widehat{v}_h|_{H^2(\widehat{K})}. \label{eq13}
 \end{equation}
A scaling argument leads to
 \begin{eqnarray*}
&& \sum_{K\in\mathcal{T}_h}\sum_{i\neq j=1}^3 \frac{h^2}{3}\int_K\frac{\partial^3u}{\partial x_i\partial x_j^2}\frac{\partial^3v_h}{\partial x_i^3}\,\mathrm{d}x_1\mathrm{d}x_2\mathrm{d}x_3+\frac{1}{2}\sum_{K\in\mathcal{T}_h}\int_K(b_9+b_{12})\frac{\partial^3v_h}{\partial x_1^3}\;\mathrm{d}x_1\mathrm{d}x_2\mathrm{d}x_3\\&&+\frac{1}{2}\sum_{K\in\mathcal{T}_h}\int_K(b_{10}+b_{13})\frac{\partial^3v_h}{\partial x_2^3}\;\mathrm{d}x_1\mathrm{d}x_2\mathrm{d}x_3+\frac{1}{2}\sum_{K\in\mathcal{T}_h}\int_K(b_{11}+b_{14})\frac{\partial^3v_h}{\partial x_3^3}\;\mathrm{d}x_1\mathrm{d}x_2\mathrm{d}x_3\\&&\leq Ch^2|u|_{H^4(\Omega)}|v_h|_h.
 \end{eqnarray*}
An application of Lemma \ref{lem11} yields
 \begin{equation}
 a_h(u-\widetilde{\Pi}_h^*u,v_h)=a_h(u-\Pi_h u,v_h)+a_h(\widetilde{R}_hu,v_h)\leq Ch^2|u|_{H^4(\Omega)}|v_h|_h. \label{eq56}
 \end{equation}
 Then, by Theorem \ref{thm8} and (\ref{eq56}), it yields
 \begin{eqnarray*}
 |\widetilde{\Pi}^*_h u-u_h|_h^2&=&a_h(\widetilde{\Pi}^*_h u-u_h,\widetilde{\Pi}^*_h u-u_h)\\&=&a_h(u-u_h,\widetilde{\Pi}^*_h u-u_h)+a_h(\widetilde{\Pi}^*_h u-u,\widetilde{\Pi}^*_h u-u_h)\\&=&[a_h(u,\widetilde{\Pi}^*_h u-u_h)-(f,\widetilde{\Pi}^*_h u-u_h)]+a_h(\widetilde{\Pi}^*_h u-u,\widetilde{\Pi}^*_h u-u_h)\\&\leq&Ch^2(||f||_{L^2(\Omega)}+|u|_{H^4(\Omega)})|\widetilde{\Pi}^*_h u-u_h|_h,
 \end{eqnarray*}
 which completes the proof.
\end{proof}
Furthermore, based on the superclose property, we can obtain the superconvergence result of the three-dimensional cubic Morley element by a proper postprocessing technique.  In order to attain the global superconvergence, we follow the idea of \cite{LinYan1996} to construct the postprocessing operator $\Pi_{3h}^3$ as follows.\\
We merge 27 adjacent elements into a macro element, $\widetilde{K}=\bigcup\limits_{i=1}^{27}K_i$, (see Figure 4), such that, in the macro element $\widetilde{K}$,
\begin{equation}
\Pi_{3h}^3w\in Q_3(\widetilde{K}),\quad \forall w \in C(\widetilde{K}).
\end{equation}
We denote $Z_{ijk},i,j,k=1,2,3,4$ as the vertices of the 27 adjacent elements. Then, the operator $\Pi_{3h}^3$ satisfies
\begin{equation}
\Pi_{3h}^3w(Z_{ijk})=w(Z_{ijk}),\quad i,j,k=1,2,3,4.
\end{equation}

\begin{figure}
\begin{center}
\setlength{\unitlength}{1.2cm}
\begin{picture}(5,5)
\put(0,0){\line(1,0){3}}
\put(0,1){\line(1,0){3}}
\put(0,2){\line(1,0){3}}
\put(0,3){\line(1,0){3}}
\put(0,0){\line(0,1){3}}
\put(1,0){\line(0,1){3}}
\put(2,0){\line(0,1){3}}
\put(3,0){\line(0,1){3}}
\put(3,0){\line(1,1){1.2}}
\put(3,1){\line(1,1){1.2}}
\put(3,2){\line(1,1){1.2}}
\put(3,3){\line(1,1){1.2}}
\put(0,3){\line(1,1){1.2}}
\put(1,3){\line(1,1){1.2}}
\put(2,3){\line(1,1){1.2}}
\put(3,3){\line(1,1){1.2}}
\put(1.2,4.2){\line(1,0){3}}
\put(0.4,3.4){\line(1,0){3}}
\put(0.8,3.8){\line(1,0){3}}
\put(4.2,1.2){\line(0,1){3}}
\put(3.4,0.4){\line(0,1){3}}
\put(3.8,0.8){\line(0,1){3}}
\end{picture}
\end{center}
\caption{macro element $\widetilde{K}$}
\end{figure}
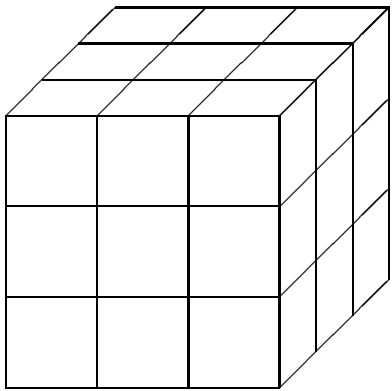
Besides, the postprocessing operator $\Pi_{3h}^3$ has the following properties
\begin{eqnarray}
&&\left\{
\begin{array}{lll}
\Pi_{3h}^3(\widetilde{\Pi}^*_h u)=\Pi_{3h}^3u,\quad \forall u\in H^4(\Omega),\\
|\Pi_{3h}^3 v_h|_h\leq C|v_h|_h,\quad \forall v_h\in V_h,\\
|u-\Pi_{3h}^3 u|_h\leq Ch^2|u|_{H^4(\Omega)},\quad \forall u\in H^4(\Omega).\label{eq63}
\end{array}
\right.\\ \nonumber
\end{eqnarray}
Then, we can get the following global superconvergent result.
\begin{thm}
Let $u\in H^4(\Omega)$, $u_h\in V_h$, be the solutions of (\ref{eq1}) and (\ref{eq5}), respectively, then it holds
\begin{equation}
|u-\Pi_{3h}^3u_h|_h\leq Ch^2\big(||f||_{L^2(\Omega)}+|u|_{H^4(\Omega)}\big).
\end{equation}
\end{thm}
\begin{proof}
It follows the properties (\ref{eq63}) and Theorem \ref{thm7} that
\begin{eqnarray}
|u-\Pi_{3h}^3u_h|_h&\leq&|u-\Pi_{3h}^3\widetilde{\Pi}^*_h u_h|_h+|\Pi_{3h}^3(\widetilde{\Pi}^*_h u-u_h)|_h\nonumber\\&\leq&|u-\Pi_{3h}^3u|_h+C|\widetilde{\Pi}^*_h u-u_h|_h\nonumber\\&\leq&Ch^2\big(||f||_{L^2(\Omega)}+|u|_{H^4(\Omega)}\big),
\end{eqnarray}
which completes the proof.
\end{proof}

\section{Numerical results}
In this section, we present some numerical results of the two-dimensional rectangular Morley element and three-dimensional cubic Morley element to demonstrate our theoretical results. Herein, we denote $r$ as the rate of convergence. For the sake of simplicity, denote\\

Err1=$|u-u_h|_h$, \quad  Err2=$|\Pi_h u-u_h|_h$,  \quad  Err3=$|\Pi_h^*u-u_h|_h$, \\

Err4=$|u-\Pi_{3h}^3\Pi_h^*u|_h$,\quad  Err5=$|\widetilde{\Pi}_h^*u-u_h|_h$,  \quad  Err6=$|u-\Pi_{3h}^3\widetilde{\Pi}_h^*u|_h$.\\

In the two-dimensional case, we choose the square domain $\Omega_1=[0,1]^2$. We partition the domain $\Omega_1$ into the uniform squares with the meshsize $h=\frac{1}{N}$ for some integer $N$.\\

$\bullet$ In the first example, we use the function $u_1(x,y)=sin^2(\pi x)sin^2(\pi y)$ as the exact solution of problem (\ref{eq45}).\\

$\bullet$ In the second example, we use the function $u_2(x,y)=x^2(1-x)^2y^2(1-y)^2$ as the exact solution of problem (\ref{eq45}). \\

The errors Err1, Err2, Err3, Err4 are computed on $\Omega_1$, the corresponding computational results of the two-dimensional rectangular Morley element are listed in Table 3 and Table 4, respectively. One can also refer to Figure 5 for logarithmic plot of the norms above-mentioned.

In the three-dimensional case, we choose the square domain $\Omega_2=[0,1]^3$. We partition the domain $\Omega_2$ into the uniform cubic meshes with the meshsize $h=\frac{1}{N}$ for some integer $N$. \\

$\bullet$ In the third example, we use the function $u_3(x,y,z)=sin^2(\pi x)sin^2(\pi y)sin^2(\pi z)$ as the exact solution of problem (\ref{eq45}).\\

$\bullet$ In the fourth example, we use the function $u_4(x,y,z)=x^2(1-x)^2y^2(1-y)^2z^2(1-z)^2$ as the exact solution of problem (\ref{eq45}). \\

The errors Err1, Err2, Err5, Err6 are computed on $\Omega_2$, the corresponding computational results of the three-dimensional cubic Morley element are listed in Table 5 and Table 6, respectively. One can also refer to Figure 6 for logarithmic plot of the norms above-mentioned.

\begin{table}[h]
 \centering
 \caption{The errors of the 2-D rectangular Morley element for $u_1(x,y)$}
 \begin{tabular}{cccccc}
 \hline
 N              &6                        &12                   &24        &48   &\\\hline
 Err1             &3.801933642  &1.848733847   &0.916356489 &0.457125924&\vspace{2mm}\\
 r                  &--- &1.040195809  &1.012556679& 1.003317321&\vspace{2mm}\\
 Err2            &1.97558386 &1.04229950   & 0.526008399 &0.263522765&\vspace{2mm}\\
 r                  &---&0.922509198&0.986612148 &0.997158239&\vspace{2mm}\\
 Err3           &1.614468104  &0.436680422   &0.111290839 &0.028096575&\vspace{2mm}\\
 r                 &--- &1.886409182 &1.972243010&1.985868661& \vspace{2mm}\\
 Err4             &3.503176938  &1.427388728   &0.359795304&0.09010039&\vspace{2mm}\\
 r                  &--- &1.295285573  &1.988130023&1.997571101&\vspace{2mm}\\ \hline
 \end{tabular}
\end{table}

\begin{table}[h]
 \centering
 \caption{The errors of the 2-D rectangular Morley element for $u_2(x,y)$}
 \begin{tabular}{cccccc}
 \hline
  N                &6                        &12                   &24        &48   &\\\hline
 Err1             &0.014701829  &0.007300451   &0.003640499 &0.001803668&\vspace{2mm}\\
 r                  &--- &1.009938149  &1.003849379&1.013202406&\vspace{2mm}\\
 Err2            &0.008054841 &0.004151273  & 0.002093528 &0.00100988&\vspace{2mm}\\
 r                  &---&0.956302299 & 0.987617597&1.051752340&\vspace{2mm}\\
 Err3           &0.006314902  &0.001727848  &0.000441448 &0.00011096&\vspace{2mm}\\
 r                 &--- &1.869784037 &1.968660896 &1.992203814& \vspace{2mm}\\
 Err4             &0.015833538 &0.003874778   &0.000963156&0.00024043&\vspace{2mm}\\
 r                  &--- &2.030798101  &2.008272264&2.002152565&\vspace{2mm}\\ \hline
 \end{tabular}
\end{table}

\begin{table}[h]
 \centering
 \caption{The errors of the 3-D rectangular Morley element for $u_3(x,y,z)$}
 \begin{tabular}{cccccc}
 \hline
 N              &6                        &12                   &24        &48   &\\\hline
 Err1             &4.167950479 &1.983259265  &0.97524877 & 0.48523741   &\vspace{2mm}\\
 r                  &--- &1.071464848  &1.02403111&  1.007079492   &\vspace{2mm}\\
 Err2            &2.569313835 &1.299194244   & 0.64723551 &0.32311494    &\vspace{2mm}\\
 r                  &---&0.9837659766 &1.00525448 &  1.002243303  &\vspace{2mm}\\
 Err5           &2.221462641  &0.611765555   &0.15663992 &  0.03948401 &\vspace{2mm}\\
 r                 &--- &1.860459095 &1.965526946 &   1.988111511    & \vspace{2mm}\\
 Err6             &3.862477845 &1.283658784   &0.319649973     & 0.07982083   &\vspace{2mm}\\
 r                  &--- &1.589264895  &2.005696886  & 2.001655786    &\vspace{2mm}\\ \hline
 \end{tabular}
\end{table}

\begin{table}[h]
 \centering
 \caption{The errors of the 3-D rectangular Morley element for $u_4(x,y,z)$}
 \begin{tabular}{cccccc}
 \hline
 N                 &6                        &12                   &24        &48   &\\\hline
 Err1             &0.001051488 &0.000509571  &0.00025198&0.00012562 &\vspace{2mm}\\
 r                  &--- &1.045077305  &1.015973946& 1.004243055 &\vspace{2mm}\\
 Err2            &0.000666808 &0.000330208   & 0.00016451 & 0.000082177  &\vspace{2mm}\\
 r                  &---&1.013896 &1.00519979 & 1.001368714    &\vspace{2mm}\\
 Err5           &0.00055831  &0.000154203  &0.000039494 &0.000009933   &\vspace{2mm}\\
 r                 &--- &1.856235564 &1.965125435 & 1.991332077      & \vspace{2mm}\\
 Err6             &0.00096663 &0.000232674   &  0.000057562   &0.000014351    &\vspace{2mm}\\
 r                  &--- &2.054653764  &2.015121382   &2.003965450  &\vspace{2mm}\\ \hline
 \end{tabular}
\end{table}

\begin{figure}[!hbt]
  \centering
  \includegraphics[width=6.5in, height=3in]{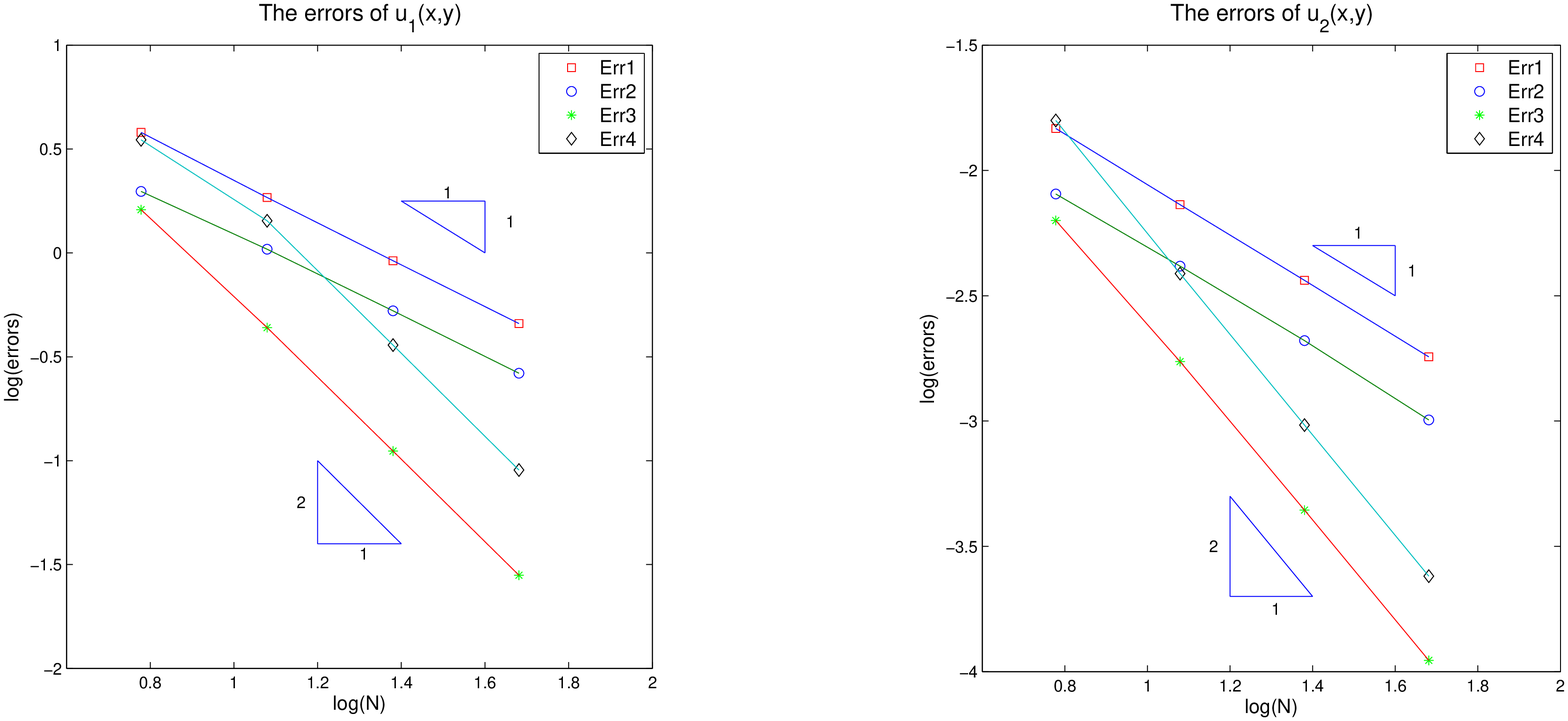}
  \caption{The errors of the 2-D rectangular Morley element for $u_1(x,y)$ and $u_2(x,y)$}

\end{figure}

\begin{figure}[!hbt]
  \centering
  \includegraphics[width=6.5in, height=3in]{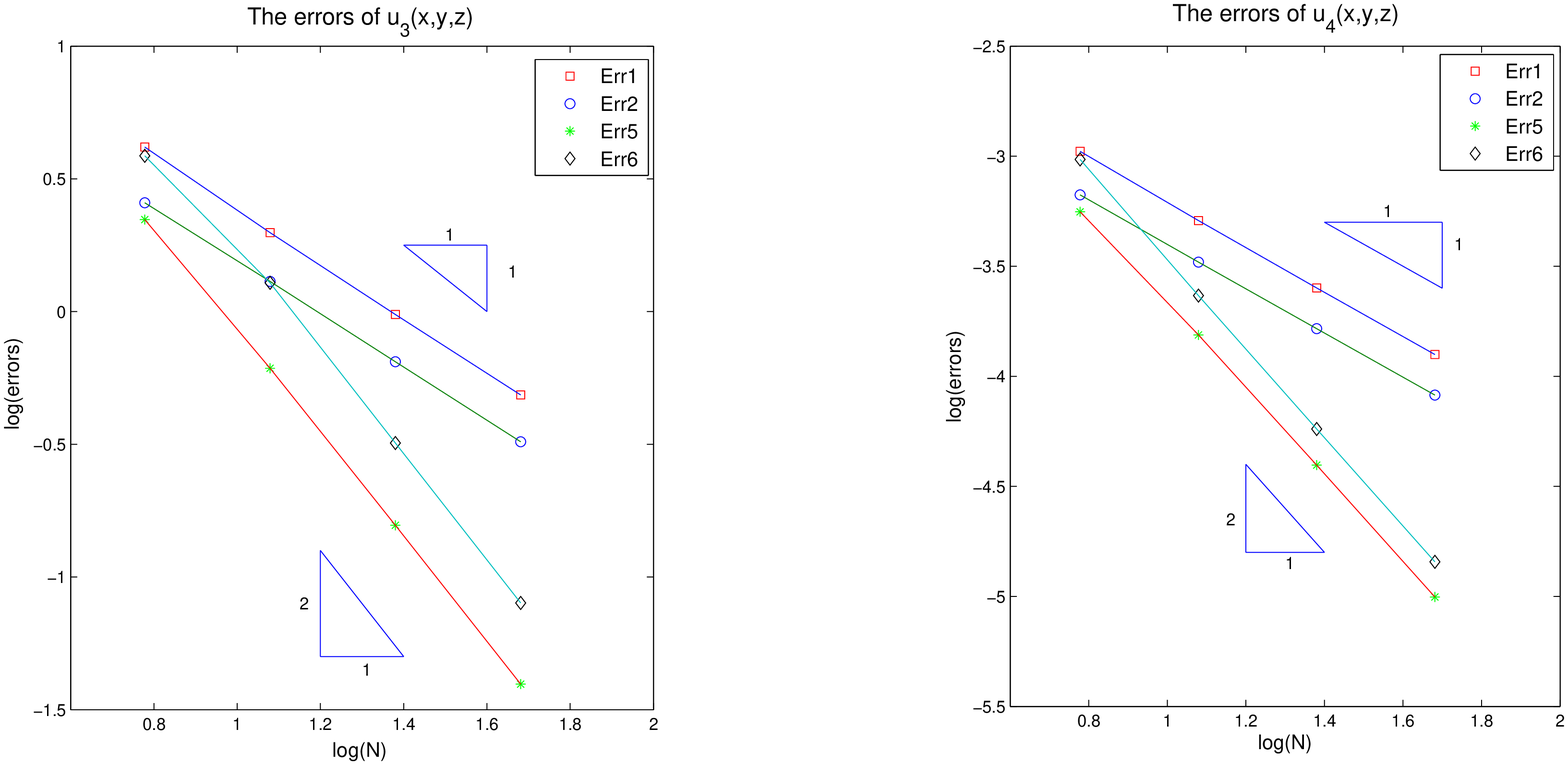}
  \caption{The errors of the 3-D rectangular Morley element for $u_3(x,y,z)$ and $u_4(x,y,z)$}

\end{figure}

From the tables and figures, we can see the superconvergent behaviors of the numerical solutions. Besides, in our examples,  the exact solution $u_i(x,y),i=1,2$, or $u_i(x,y,z),i=3,4$, don't satisfy the boundary condition $\frac{\partial^3 u_i}{\partial n^3}=0, i=1,\cdots,4$, which are need for superconvergence of second order in two-dimensional case \cite{HuMa, MaoShi}. However, our results still have the superconvergent property, which are coincide with our theoretical analysis.

\end{document}